\documentclass[12pt]{article}

\usepackage{amsthm,amssymb}
\makeatletter
\ifx\volno\undefined\def\volno{27}\fi
\ifx\volyear\undefined\def\volyear{2020}\fi
\ifx\papno\undefined\def\papno{P00}\fi
\ifx\papid\undefined\def\papid{abcd}\fi

\newcommand{\MSC}[1]{\def\the@MSC{\ignorespaces#1}}

\renewenvironment{abstract}{%
  \small
      \begin{center}%
	{\bfseries \abstractname\vspace{-.5em}\vspace{\z@}}%
      \end{center}%
      \quotation}%
  {\par\smallskip\noindent
   \textbf{Mathematics Subject Classifications: }\the@MSC\endquotation}

\renewcommand\title[1]{\gdef\@title{\reset@font\Large\bfseries #1}}
\renewcommand\section{\@startsection {section}{1}{\z@}%
                                   {-3.5ex \@plus -1ex \@minus -.2ex}%
                                   {2.3ex \@plus.2ex}%
                                   {\normalfont\large\bfseries}}
\renewcommand\subsection{\@startsection{subsection}{2}{\z@}%
                                     {-3ex\@plus -1ex \@minus -.2ex}%
                                     {1.5ex \@plus .2ex}%
                                     {\normalfont\normalsize\bfseries}}
\renewcommand\subsubsection{\@startsection{subsubsection}{3}{\z@}%
                                     {-2.5ex\@plus -1ex \@minus -.2ex}%
                                     {1.5ex \@plus .2ex}%
                                     {\normalfont\normalsize\bfseries}}

\renewcommand\paragraph{\@startsection{paragraph}{4}{\z@}%
                                     {2ex \@plus.5ex \@minus.2ex}%
                                     {-1em}%
                                     {\normalfont\normalsize\bfseries}}

\renewcommand\subparagraph{\@startsection{subparagraph}{5}{\parindent}%
                                     {2ex \@plus.5ex \@minus .2ex}%
                                     {-1em}%
                                     {\normalfont\normalsize\bfseries}}

\renewcommand{\ps@plain}{%
\renewcommand{\@oddfoot}{\hfil\thepage}}

\g@addto@macro\bfseries{\boldmath}

  \newlength{\BiblioSpacing}
  \renewenvironment{thebibliography}[1]{%
    \begin{oldthebibliography}{#1}%
      \setlength{\parskip}{\BiblioSpacing}
      \setlength{\itemsep}{\BiblioSpacing}
  }%
  {%
    \end{oldthebibliography}%
  }

\theoremstyle{plain}
\newtheorem{theorem}{Theorem}

\newtheorem{proposition}[theorem]{Proposition}

\theoremstyle{definition}
\newtheorem{definition}[theorem]{Definition}
\newtheorem{example}[theorem]{Example}

\theoremstyle{remark}
\newtheorem{remark}[theorem]{Remark}

\makeatother

\usepackage{amsmath}
\MSC{03E05}
\title{$m$-submultisets and $m$-permutations\\ of multisets elements}

\author{Oleksandr Makhnei \qquad  Roman Zatorskii\\
\small Faculty of Mathematics and Computer Science\\[-0.8ex]
\small Vasyl Stefanyk Precarpathian National University\\[-0.8ex]
\small Ivano-Frankivsk, Ukraine\\
\small\tt oleksandr.makhnei@pnu.edu.ua\qquad roman.zatorskii@pnu.edu.ua}

\begin{document}

\maketitle

\begin{abstract}
The article contains some important classes of multisets.
Combinatorial proofs of problems on the number of
$m$-submultisets and $m$-permutations of multiset elements are
considered and effective algorithms for their calculation are
given. In particular, the Pascal triangle is generalized in the
case of multisets.
\end{abstract}

\section{Introduction}

The first spontaneous combinatorial studies of permutations of
multisets, apparently, begin with the studies of the Indian
mathematician Bh\={a}skara II (1150). The polynomial formula
for the number of all permutations of an arbitrary multiset was
considered by Jean Prestet in the paper \cite{Prestet}.

In discrete mathematics, problems of investigating sets of
objects with identical objects often arise. Therefore, from the
middle of the last century the concept of multiset (see
\cite{Stanley1, Stanley2}) begins to gain more and more weight.
Since the multiset is a natural generalization of the set, the
problems of generalization of the classical results of
combinatorics of finite sets naturally arise. Thus, in the
paper \cite{Gr-Kn}, Green and Kleitman, in fact, consider the
problem of calculating the number of $m$-submultisets of a
multiset. However, in the general case, few problems are
solved. As a rule, authors are limited to considering only some
partial but very important classes of multisets.

In the papers \cite{Cartier_Foata,Foata}, Dominique Foata
introduced the concept of ``joining product''
$\alpha\top\beta$, which extended a number of known results
concerning ordinary permutations of sets to the case of
multisets. In the book \cite{Knuth3}, Donald Knuth develops
combinatorial techniques for multisets. Using the theorem that
each permutation of a multiset can be written as
$$\sigma_1\top\sigma_2\top\ldots\top\sigma_t,\; t\geqslant 0,$$
where $\sigma_j$ are cycles such that their elements are not
repeated, Knuth gives examples of enumeration of permutations
of multisets with some restrictions.

The paper \cite{Savage_Wilf} is very useful from an applied
point of view.

The so-called nondecreasing series in the permutations of
multisets (see \cite{David_Barton,Mallows}) have important
applications in the study of ``order statistics''. In the case
of a constant multiset $\{1^p,2^p,\ldots,m^p\}$ in his paper
\cite[212--213]{MacMahon}, Percy MacMahon showed such that the
number of permutations with $k+1$ series is equal to the number
of permutations with $mp-p-k+1$ series. Also, by the generatrix
method, MacMahon proved such that the number of permutations of
the multiset $1^{n_1,},2^{n_2},\ldots,m^{n_m}$ with $k$ series
is equal to
$$\sum_{j=0}^k(-1)^j\binom{n+1}{j}\binom{n_1-1+k-j}{n_1}\binom{n_2-1+k-j}{n_2}\cdot\ldots\cdot\binom{n_m-1+k-j}{n_m}\!,$$
where $n=n_1+n_2+\ldots+n_m.$

An interesting approach for enumerating submultisets of
multisets is proposed in \cite{Hage}.

Sometimes a continual apparatus is used to solve discrete
mathematics problems. For instance, in \cite{Goc1}, using the
generatrix method, Goculenko proved an integral formula for
calculating the number of $m$-submultisets of the given
multiset
$$|C^m(A)|=\frac{1}{2\pi}\int_{-\pi}^{\pi}exp(-im\varphi)\prod_{j=1}^n\frac{exp\{i(k_j+1)\varphi\}-1}{exp\{i\varphi\}-1},$$
where $i=\sqrt{-1}.$ In \cite{Goc1} the problem for
$m$-submultisets of a multiset is also somewhat generalized.

This paper contains some important classes of multisets.
Combinatorial proofs of problems on the number of
$m$-submultisets and $m$-permutations of multiset elements are
considered and effective algorithms for their calculation are
given. In particular, the Pascal triangle is generalized in the
case of multisets.

\section{Auxiliary concepts}

The multiset $A$ means an arbitrary disordered set of elements
of some set $[A]$, which we call the base of this multiset.
Therefore, an arbitrary multiset can be written in the
canonical form
 \begin{equation}\label{multy}A=\{a_1^{k_1},a_2^{k_2},\ldots,a_n^{k_n}\},
 \end{equation}
where $[A]=\{a_1,a_2,\ldots,a_n\}$ and indices $k_i$ of
elements $a_i$ indicate the multiplicity of occurrence of the
element $a_i$ to the multiset $A.$ We can assume without loss
of generality that $k_1\geqslant k_2\geqslant \ldots \geqslant
k_n.$

The multiset $A^{'}=\{k_1,k_2,\ldots,k_n\}$ of indices of
multiset (\ref{multy}) is called its primary specification.
Suppose the primary specification $A^{'}$ of multiset
(\ref{multy}) is represented in the canonical form
$$A^{'}=\{1^{\lambda_1},2^{\lambda_2},\ldots,r^{\lambda_r}\},$$
where
\begin{equation}\label{r=max_k}
r=\max(k_1,k_2,\ldots,k_n);
\end{equation}
then the multiset of its indices
$A^{''}=\{\lambda_1,\lambda_2,\ldots,\lambda_r\}$ is called the
secondary specification of the multiset $A.$

If $i$ does not belong to the multiset $A^{'},$ then we assume
that $\lambda_i=0.$ Note that for the secondary specification
of multiset (\ref{multy}) we have the equality
\begin{equation}\label{l_1+2l_2+...=|A|}
|A|=\lambda_1+2\lambda_2+\ldots+r\lambda_r.
\end{equation}

Multisets $A=\{a_1^{k_1},a_2^{k_2},\ldots,a_n^{k_n}\}$ and
$\overline{A}=\{a_1^{\overline{k_1}},a_2^{\overline{k_2}},\ldots,a_r^{\overline{k_r}}\}$
are called the adjoint multisets if
\begin{equation}\label{-k=k}
\overline{k_i}=\left|\{k_j:k_j\geqslant i\}\right|=\left|\{j:k_j\geqslant i\}\right|,\; i=1,2,\ldots,r,\;
j=1,2,\ldots,n.
\end{equation}
Here $r$ is given by equality (\ref{r=max_k}).

Let us remark that $\overline{k_i}$ has a certain combinatorial
meaning. Namely $\overline{k_i}$ is the maximum number of
groups of $i$ identical elements that can be chosen from
multiset (\ref{multy}).

If the equality $\overline{A}=A$ holds true, then the multiset
$A$ is called the multiset with a self-adjoint primary
specification or the self-adjoint multiset.

If $A$ and $\overline{A}$ are the adjoint multisets and
$$A^{'}=\{k_1,\ldots,k_n\},\;
A^{''}=\{\lambda_1,\ldots,\lambda_r\},\;
\overline{A}^{'}=\{\overline{k_1},\ldots,\overline{k_r}\},
$$
$$(\overline{A})^{''}=\{\overline{\lambda_1},\ldots,\overline{\lambda_n}\}, \,\,r=\max(k_1,k_2,\ldots,k_n),$$
then between the elements of their specifications, in addition
to relationship (\ref{-k=k}), you can give 11 next
relationships. 

 \begin{equation}\label{k=lambda}k_i=\left|\{\lambda_j+\ldots+\lambda_r:\lambda_j+\ldots+\lambda_r\geqslant i\}\right|,\; i=1,\ldots,n,\; j=1,\ldots,r,
 \end{equation}
 \begin{equation}\label{-lambda=-k}\overline{\lambda_i}=\left|\{j:\overline{k_j}=i\}\right|,\; i=1,\ldots,n,\; j=1,\ldots,r,
 \end{equation}
 \begin{equation}\label{-k=-lambda}
 \overline{k_i}=\left|\{\overline{\lambda_j}+\ldots+\overline{\lambda_r}:\overline{\lambda_j}+\ldots+
 \overline{\lambda_r}\geqslant i\}\right|,\; i=1,\ldots,r,\; j=1,\ldots,n,
 \end{equation}
\begin{equation}\label{lambda=k}\lambda_i=\left|\{j:k_j=i\}\right|,\; i=1,\ldots,r,\; j=1,\ldots,n,
 \end{equation}
 \begin{equation}\label{k=-k}k_i=\left|\{\overline{k_j}:\overline{k_j}\geqslant i\}\right|,\; i=1,\ldots,n,\; j=1,\ldots,r,
 \end{equation}
\begin{equation}\label{-lambda=lambda}\overline{\lambda_i}=\left|\{j:\lambda_j+\ldots+\lambda_r=i\}\right|,\; i=1,\ldots,n,\; j=1,\ldots,r,
 \end{equation}
 \begin{equation}\label{lambda=-lambda}\lambda_i=\left|\{j:\overline{\lambda_j}+\ldots+\overline{\lambda_r}=i\}\right|,\; i=1,\ldots,r,\; j=1,\ldots,n,
 \end{equation}
\begin{equation}\label{-lambda=k}M\cdot \overline{\lambda}=k,
 \end{equation}
 \begin{equation}\label{k=-lambda}M^{-1}\cdot k=\overline{\lambda},
 \end{equation}
\begin{equation}\label{lambda=-k}M\cdot \lambda=\overline{k},
 \end{equation}
\begin{equation}\label{-k=lambda}M^{-1}\cdot
\overline{k}=\lambda.
 \end{equation}

In equalities (\ref{-lambda=k}) and (\ref{k=-lambda}) $k$ and
$\overline{\lambda}$ are $n$-dimensional column vectors such
that their coordinates coincide with the elements of the
specifications $k(A)$ and $k^2(\overline{A})$ accordingly. In
equalities (\ref{-lambda=k}) and (\ref{k=-lambda}) $M$ and
$M^{-1}$ are square matrices of order $n$ of the next form
$$M=\left(
  \begin{array}{ccccc}
    1 & 1 & \cdots & 1 & 1 \\
    0 & 1 & \cdots & 1 & 1 \\
    \vdots & \cdots & \cdots & \cdots & \vdots \\
    0 & 0& \cdots & 1 & 1 \\
    0 & 0 & \cdots & 0 & 1 \\
  \end{array}
\right)\!,\quad
M^{-1}=\left(
  \begin{array}{ccccc}
    1 & -1 & \cdots & 0 & 0 \\
    0 & 1 & \cdots & 0 & 0 \\
    \vdots & \cdots & \cdots & \cdots & \vdots \\
    0 & 0& \cdots & 1 & -1 \\
    0 & 0 & \cdots & 0 & 1 \\
  \end{array}
\right)\!.$$

In equalities (\ref{lambda=-k}) and (\ref{-k=lambda}) $\lambda$
and $\overline{k}$ are similar $r$-dimensional column vectors,
$M$ and $M^{-1}$ are similar matrices of order $r$.
     		
\begin{remark}Since $k(\overline{\overline{A}})=k(A)$, it follows
that formulas (\ref{-k=k}), (\ref{k=lambda}),
(\ref{-lambda=-k}), (\ref{-lambda=lambda}), (\ref{-lambda=k}),
(\ref{k=-lambda}) are analogous to formulas (\ref{k=-k}),
(\ref{-k=-lambda}), (\ref{lambda=k}), (\ref{lambda=-lambda}),
(\ref{lambda=-k}), (\ref{-k=lambda}) correspondingly. In fact,
formulas (\ref{-lambda=k}), (\ref{k=-lambda}) establish the
one-to-one correspondence between the sets of solutions of
equation (\ref{l_1+2l_2+...=|A|}) and the equation
$|A|=\overline{\lambda_1}+2\overline{\lambda_2}+\ldots+n\overline{\lambda_n}$,
which is analogous to equation (\ref{l_1+2l_2+...=|A|}). A
similar conclusion can be made for formulas (\ref{lambda=-k}),
(\ref{-k=lambda}).
\end{remark}

Finally, we give a well-known statement about a cardinality of
multiboolean of multiset (\ref{multy}).

\begin{proposition}\label{|A^k|}If
$$A=\{a_1^{k_1},a_2^{k_2},\ldots,a_n^{k_n}\}$$ and $C(A)$ is a set of submultisets of the multiset
$A$, then
\begin{equation}\label{C(A)}\left|C(A)\right|=\prod_{i=1}^n(k_i+1).\end{equation}
\end{proposition}

\section{Some classes of multisets and their specifications}

1. 	\emph{The multiset with a positive integer function of a
natural argument} is the multiset of the form
 \begin{equation}\label{g:N-N}A=\{a_1^{g(1)},a_2^{g(2)},\ldots,a_n^{g(n)}\},
 \end{equation}
where $g:N\rightarrow N$ is some nondecreasing function that
satisfies the inequality $g(i)\geqslant i$ for all $i\in N$.

2. \emph{The multiset with a continuous function} $f$ is the
multiset of the form
\begin{equation}\label{f:D-E}A=\{[f(1)],[f(2)],\ldots,[f(n)]\},
 \end{equation}
where $f$ is some continuous increasing function
$$f:D\rightarrow E,\; D=[1,n],\; E\supseteq [1,[f(n)]]$$ that
satisfies the inequality $f(x)\geqslant x,$ $[\:]$ is an
integer part of the number. Specification (\ref{f:D-E}) is a
partial case of multiset (\ref{g:N-N}).

For example, for the function $f=\exp(x)$ and $n = 5$ the first
derivative of the multiset has the form
$$A'=\{2,7,20,54,148\}.$$

3. \emph{The linear multiset} is the multiset that has the form
\begin{equation}\label{specyf_lin}A=\{a_1^{p+q},a_2^{p+2q},\ldots,a_n^{p+nq}\},
 \end{equation}
where $p\in N_0$, $q\in Z$ and $1\leqslant p+q.$

4. \emph{The constant multiset} is the multiset that has the
form
\begin{equation}\label{specyf_const}A=\{a_1^{q},a_2^{q},\ldots,a_n^{q}\},
 \end{equation}
where $q\geqslant 1.$

5. \emph{The multiset with repetitions without restrictions} is
the specification that has the form
\begin{equation}\label{specyf_infty}A=\{a_1^{\infty},a_2^{\infty},\ldots,a_n^{\infty}\}.
 \end{equation}

Finally, we give an example of another class of multisets such
that a number of $m$-submultisets is calculated relatively
simply.
\begin{equation}\label{specyf_step}A=\{a_1^{2^{k_1}-1},a_2^{2^{k_2}-1},\ldots,a_n^{2^{k_n}-1}\},\quad k_1\leqslant k_2\leqslant\ldots\leqslant k_n.
 \end{equation}

\section{Number of $m$-submultisets of a multiset}

\begin{definition}\label{def.Cm(A)}The set
\begin{equation}\label{Cm(A)}C_m(A)=\{B\subseteq A:|B|=m\}
\end{equation}
of all $m$-submultisets of the multiset
$A=\{a_1^{k_1},\ldots,a_n^{k_n}\}$ is called the set of
$m$-combinations of elements of this multiset.
\end{definition}

To denote the cardinality of set (\ref{Cm(A)}) we use the
notation
\begin{equation}\label{|Cm(A)|}|C_m(A)|=\binom{k_1k_2\ldots
k_n}{m},
\end{equation}
which was proposed in the \cite{Grin-Kleit}.

For some specifications of the multiset $A$ the cardinality of
the set has been considered formerly. In particular, for
$n$-element sets the classical formula
\begin{equation}\label{C_n^m}\binom{\underbrace{11\ldots
1}_{n}}{m}=\frac{n!}{m!(n-m)!}
\end{equation}
is known.

For a multiset with repetitions without restrictions it is
known such that the formula
\begin{equation}\label{infty,infty}\binom{\underbrace{\infty\,\infty\,\ldots
\,\infty}_{n}}{m}=\frac{(n+m-1)!}{m!(n-1)!}
\end{equation}
is valid.

\begin{theorem}\label{theo.C_m(A)}
{\sloppy The number of $m$-submultisets ($m$-combinations) of
the multiset $A=\{a_1^{k_1},\ldots,a_n^{k_n}\}$ is equal to
\begin{equation}\label{eq.C_m(A)}
\binom{k_1\,k_2\,\ldots\,k_n}{m}= |C_m(A)|=\sum_{\lambda\in \Lambda_m(A)}\prod_{j=1}^s\binom{\overline{k_j}-\sum_{i=j+1}^s\lambda_i}{\lambda_j},
\end{equation}
where $\Lambda_m(A)$ is the set of those solutions of the
equation
\begin{equation}\label{sumilambda=m}\sum_{i=1}^si\lambda_i=m
\end{equation}
that satisfy the inequalities
\begin{equation}\label{sumilambda<kj}\sum_{i=j}^s\lambda_i\leqslant \overline{k_j},\; j=1,\ldots,s,
\end{equation}
where
$$s=\min(m,r),\;\, r=\max\{k_i\},\;\, i=1,\ldots,n,$$
$\overline{k_j}$ is the $j$th element of specification
(\ref{-k=k}), which is adjoint to the primary specification of
the multiset $A.$}
\end{theorem}

\begin{proof} From the definition of the set $\Lambda_m(A)$ it follows that this set satisfies the conditions:

1) $\forall B\in C_m(A)\Rightarrow B^{''}\in\Lambda_m(A);$

2) $\forall \lambda\in\Lambda_m(A)\Rightarrow \exists B\in
C_m(A): B^{''}=\lambda.$

Let us prove that the set $\Lambda_m(A)$ consists of all
integer non-negative solutions of equation (\ref{sumilambda=m})
that satisfy inequalities (\ref{sumilambda=m}).

{\sloppy Indeed, let $B$ be some multiset that belongs to set
(\ref{Cm(A)}) and
$\lambda=\{\lambda_1,\lambda_2,\ldots,\lambda_p\}=B^{'}.$ Since
$|B|=m,$ it is obvious that the elements of this secondary
specification satisfy equation (\ref{sumilambda=m}). The truth
of inequalities (\ref{sumilambda<kj}) for solutions of this
equation follows from the inequalities $k_{x}(B)\leqslant
k_x(A)$, $x\in [B]$, where the symbol $k_x(B)$ denotes the
multiplicity of occurrence of the element $x$ to the multiset
$B$.

}

Let $\lambda=\{\lambda_1,\ldots,\lambda_s\}$ be some solution
of equation (\ref{sumilambda=m}) that satisfies inequalities
(\ref{sumilambda<kj}). We construct a multiset $B\in C_m(A)$
such that $B^{''}=\lambda.$ Let us start by selecting from the
multiset $A$ $\lambda_s$ different groups of $s$ identical
elements. This can always be done because $\lambda_s\leqslant
\overline{k_s}$ due to (\ref{-k=k}). Suppose we have already
selected $\sum_{i=j+1}^s\lambda_i$ different groups of elements
such that each group consists of at least $j+1$ identical
elements. Let $\overline{k_j}$ be the maximum number of groups
of $j$ identical elements that can be selected from the
multiset $A$; then there are
$$\overline{k_j}-\sum_{i=j+1}^s\lambda_i$$
groups of $j$ identical elements in each group, in addition to
other groups, in the multiset $A$ after selecting from this
multiset of the above groups of elements. Thus, the selection
of the following $\lambda_j$ groups of identical elements from
the multiset $A$ ensures the fulfillment of inequalities
(\ref{sumilambda<kj}).

If every secondary specification from the set $\Lambda_m(A)$ is
assigned a non-empty set
\begin{equation}\label{C_lambda^m(A)}
C^{\lambda}_m(A)=\{B\in C_m(A):B^{''}=\{\lambda_1,\ldots,\lambda_s\}\}
\end{equation}
of the multisets from the set $C_m(A)$, then set
(\ref{C_lambda^m(A)}) for $\lambda\in \Lambda_m(A)$ forms a
partition of the set $C_m(A)$. Under this condition the
equality
\begin{equation}\label{|C^m(A)|}
|C_m(A)|=\sum_{\lambda\in \Lambda_m(A)}|C^{\lambda}_m(A)|
\end{equation}
is valid. Let us find the cardinality of set
(\ref{C_lambda^m(A)}). It has already been determined such that
the multiset $A$ contains
$$\overline{k_j}-\sum_{i=j+1}^s\lambda_i$$ groups of $j$
identical elements after selecting from the multiset $A$ of all
groups of identical elements that consist of at least $j + 1$
identical elements. Therefore, there is exactly
$$\binom{\overline{k_j}-\sum_{i=j+1}^s\lambda_i}{\lambda_j}$$
different choices for these groups from the multiset $A$. The
number of all elements belonging to the set
(\ref{C_lambda^m(A)}) is equal to
\begin{equation}\label{|C_lambda^m(A)|}
|C^{\lambda}_m(A)|=\Pi_{j=1}^s\binom{\overline{k_j}-\sum_{i=j+1}^s\lambda_i}{\lambda_j}
\end{equation}
by the combinatorial rule of the product. Here and then we have
$\sum_{i>s}^s\lambda_i=0$. Note that if the inequalities
$k_1\geqslant k_2\geqslant\ldots\geqslant k_n$ are fulfilled,
then the elements of specification $\overline{k(A)}$, in
addition to the relation (\ref{-k=k}), can be calculated
according to one of the following formulas:
\begin{equation}\label{-k_j}
\overline{k_j}=n-k^{-1}(j)+1,\; j=1,\ldots,k_n,
\end{equation}
\begin{equation}\label{-k_j=sum}\overline{k_j}=\sum_{i=j}^{k_n}\lambda_i,\; j=1,\ldots,k_n,
\end{equation}
where $\lambda_i\in A^{''}$,
\begin{equation}\label{k^-1(j)}
k^{-1}(j)=\min\{i:k_i\geqslant j\}
\end{equation}
is the minimum preimage of those elements of the primary
specification $A^{'}$ that are not less than $j$. Formula
(\ref{-k_j=sum}) follows from relation (\ref{lambda=-k}). Now
from (\ref{|C^m(A)|}) and (\ref{|C_lambda^m(A)|}) it follows
that formula (\ref{|Cm(A)|}) is valid.
\end{proof}

\begin{example}\label{exam.|C^m(A)|}
Calculate the number of all $6$-submultisets of the multiset
$$A=\{a_1^5,a_2^5,a_3^5,a_4^3,a_5^3,a_6^3,a_7^3,a_8^2,a_9^2,a_{10}^1,a_{11}^1,a_{12}^1,a_{13}^1\}.$$

Here $n=13$, $m=6$, $r=5$, $s=\min(5,6)=5$. We get the elements
of the specification $\overline{k(A)}$ from relations
(\ref{-k=k}):
$$\overline{k_1}=13,\; \overline{k_2}=9,\; \overline{k_3}=7,\; \overline{k_4}=3,\; \overline{k_5}=3.$$
To find the elements of the set $\Lambda^m(A)$ we seek all
solutions of the equation
\begin{equation}\label{exam.lambda=6}\lambda_1+2\lambda_2+3\lambda_3+4\lambda_4+5\lambda_5=6.
\end{equation}
There are ten solutions of this equation:
$$(6,0,0,0,0), (4,1,0,0,0), (3,0,1,0,0), (2,2,0,0,0), (2,0,0,1,0), (1,1,1,0,0),$$
$$(1,0,0,0,1), (0,3,0,0,0), (0,1,0,1,0), (0,0,2,0,0).$$
Moreover, all these solutions satisfy the inequalities
\begin{eqnarray*}
&&\lambda_1+\lambda_2+\lambda_3+\lambda_4+\lambda_5\leqslant 13,\\
&&\lambda_2+\lambda_3+\lambda_4+\lambda_5\leqslant 9,\\
&& \lambda_3+\lambda_4+\lambda_5\leqslant 7,\\
&& \lambda_4+\lambda_5\leqslant 3,\\&& \lambda_5\leqslant 3.
\end{eqnarray*}

For each solution of equation (\ref{exam.lambda=6}) we
calculate the product (\ref{|C_lambda^m(A)|}) and seek the sum
of these products:
$$C_6(A)=\binom{13}{6}+\binom{12}{4}\cdot\binom{9}{1}+\binom{12}{3}\cdot\binom{7}{1}+\binom{11}{2}\cdot\binom{9}{2}+
\binom{12}{2}\cdot\binom{3}{1}+$$
$$\binom{11}{1}\cdot\binom{8}{1}\cdot\binom{7}{1}+\binom{12}{1}\cdot\binom{3}{1}+
\binom{9}{3}+\binom{8}{1}\cdot\binom{3}{1}+\binom{7}{2}=$$
$$=1716+4455+1540+1980+198+616+36+84+24+21=10670.$$
\end{example}

Now we calculate the number of all $m$-submultisets of the
multiset whose primary specification is a positive integer
function of a natural argument (\ref{g:N-N}).

\begin{theorem}\label{theo.g}
Suppose the multiset $A=\{a_1^{k_1},\ldots,a_n^{k_n}\}$ has the
primary specification of the form
$$A^{'}=\{g(1),g(2),\ldots,g(n)\}$$ and $g(i)\geqslant i,
i=1,2,\ldots,n$; then the equality
\begin{equation}\label{g}|C_m(A)|=\sum_{\lambda_1+\ldots+m\lambda_m=m}\prod_{j=1}^m\binom{n-g^{-1}(j)-\sum_{i=j+1}^m\lambda_i+1}{\lambda_j}
\end{equation}
is fulfilled for $m\leqslant n$, where
$$g^{-1}(j)=min\{i:g(i)\geqslant j\},\; j=1,\ldots,m.$$
\end{theorem}

\begin{proof}
First note that since the inequalities $g(n)\geqslant
n\geqslant m$, we have $s=min(m,g(n))=m$. Therefore equation
(\ref{sumilambda=m}) and inequalities (\ref{sumilambda<kj})
have the form
\begin{equation}\label{lambda+mlambda=m}\lambda_1+\ldots+m\lambda_m=m;
\end{equation}
\begin{equation}\label{sum^m<-k}\sum_{i=j}^m\lambda_i\leqslant \overline{k_j}, j=1,\ldots,m.
\end{equation}

We prove that each solution of equation
(\ref{lambda+mlambda=m}) satisfies inequalities
(\ref{sum^m<-k}). By $\Lambda$ denote the set of solutions of
equation (\ref{lambda+mlambda=m}). From the obvious
inequalities
$$\sum_{i=j}^m\lambda_i\leqslant\max_{\Lambda}\left\{\sum_{i=j}^m\lambda_i\right\}
\leqslant \left\lfloor\frac{m}{j}\right\rfloor, \;j=1,\ldots,m,$$
$$\min\{i:g(i)\geqslant j\}\leqslant j, \;j=1,\ldots,m$$
it follows that to prove the statement it is enough to prove
the validity of inequalities
\begin{equation}\label{m/j}
\left\lfloor\frac{m}{j}\right\rfloor\leqslant n-j+1, \; j=1,\ldots,m.
\end{equation}
Inequalities (\ref{m/j}) can be proved by induction on $n$.

Thus from (\ref{-k_j}) and (\ref{k^-1(j)}) it follows that
equality (\ref{g}) holds true due to Theorem \ref{theo.C_m(A)}.
\end{proof}

\begin{example}\label{exam.g}
Suppose $A=\{a_1^1,a_2^3,a_3^5,a_4^7,a_5^9\}$; then
$g(i)=2i-1\geqslant i.$ We shall find $C_4(A).$

Here $n=5$, $m=4$ and the equation
$\lambda_1+2\lambda_2+3\lambda_3+4\lambda_4=4$ has $5$
solutions:
$$(4,0,0,0), (2,1,0,0), (1,0,1,0), (0,2,0,0), (0,0,0,1).$$
We have:
\begin{eqnarray*}
&g^{-1}(1)=\min\{i:2i-1\geqslant 1\}=1;\\
&g^{-1}(2)=\min\{i:2i-1\geqslant 2\}=2;\\
&g^{-1}(3)=\min\{i:2i-1\geqslant 3\}=2;\\
&g^{-1}(4)=\min\{i:2i-1\geqslant 4\}=3.\\
\end{eqnarray*}
Therefore,
$$|C_4(A)|=\sum_{\lambda_1+2\lambda_2+3\lambda_3+4\lambda_4=4}\prod_{j=1}^4\binom{n-g^{-1}(j)-\sum_{i=j+1}^4\lambda_i+1}{\lambda_j}=$$
$$=\binom{5}{4}+\binom{4}{2}\cdot\binom{4}{1}+\binom{4}{1}\cdot\binom{4}{1}+\binom{4}{2}+\binom{3}{1}=5+24+16+6+3=54.$$
\end{example}

If the primary specification of a multiset is given by some
continuous function $f(x)$, then the following theorem is
useful for calculation of the number of all its
$m$-submultisets.

\begin{theorem}\label{theo.f(x)}
Suppose the primary specification of the multiset
$A=\{a_1^{k_1},\ldots,a_n^{k_n}\}$ has the form
$A^{'}=\{\lfloor f(1)\rfloor,\lfloor f(2)\rfloor,\ldots,\lfloor
f(n\rfloor\},$ where
$$f:D\rightarrow E, \;D=[1,n],
\;E\supseteq [1,\lfloor f(n)\rfloor]$$ is some continuous
increasing function. Then the formula
\begin{equation}\label{[f(n)]}
|C_m(A)|=\sum_{\lambda_1+\ldots+m\lambda_m=m}\prod_{j=1}^m\binom{\lfloor n-\max(f^{-1}(j),1)\rfloor-\sum_{i=j+1}^m\lambda_i}{\lambda_j}
\end{equation}
is fulfilled for $m\leqslant n$.
\end{theorem}

\begin{proof}
Since the function $f$ is continuous and increases in its
domain, we see that  the equality $\min\{i:f(i)\geqslant
j\}=f^{-1}(j)$ holds true for all $j\geqslant 1$. Hence, we
obtain the equality
\begin{equation}\label{min}
\min\{i:\lfloor f(i)\rfloor\geqslant j\}=\begin{cases}\max(1,f^{-1}(j)), \,f^{-1}(j)\in Z_{D},\\
\max(1,\lfloor f^{-1}(j)\rfloor +1), \,f^{-1}(j)\neq Z_D,\end{cases}
\end{equation}
where $Z_D=D\cap N.$ Therefore the equality $n-\min\{i:\lfloor
f(i)\rfloor\geqslant j\}=\lfloor n-\max(1,f^{-1}(j))\rfloor$ is
valid and we have the equality
$$\overline{k_j}=n-\min\{i:k_i\geqslant j\}+1=\lfloor n-\max(1,f^{-1}(j))\rfloor +1.$$

Since the inequalities
$$\max_{\lambda_1+\ldots+m\lambda_m=m}(\lambda_j+\ldots+\lambda_m)\leqslant\left\lfloor\frac{m}{j}\right\rfloor$$
and $f^{-1}(j)\leqslant j$ are fulfilled for all
$j=1,\ldots,m$, we see that inequality (\ref{sum^m<-k}) is
equivalent to inequality (\ref{m/j}). The proof of this theorem
is finished with similar reasoning to the reasoning over the
proof of theorem \ref{theo.g}.
\end{proof}

\begin{example}\label{exam.f}
Suppose in the multiset
$$A=\{a_1^{k_1},a_2^{k_2},a_3^{k_3},a_4^{k_4},a_5^{k_5},a_6^{k_6}\}$$
the primary specification is given by the continuous function
$f(x)=\sqrt{x}$ on the interval $[1,6]$, i. e.,
$$k_i=\lfloor \sqrt{i}\rfloor, \;i=1,2,3,4,5,6.$$
Then
$$k_1=1,\; k_2=1,\; k_3=1,\; k_4=2,\; k_5=2,\; k_6=2.$$

Find, for example, the number of all $5$-submultisets of the
given multiset. We have $7$ solutions of the equation
$$\lambda_1+2\lambda_2+3\lambda_3+4\lambda_4+5\lambda_5=5:$$
$$(5,0,0,0,0),(3,1,0,0,0),(2,0,1,0,0),(1,2,0,0,0),(1,0,0,1,0),$$
$$(0,1,1,0,0),(0,0,0,0,1).$$
Since $f^{-1}(x)=\min\{i:f(i)\geqslant j\}=j^2$, we have
$\lfloor n-\max(f^{-1}(j),1)\rfloor=n-j^2.$ Therefore,
$$|C_5(A)|=\sum_{\lambda_1+\ldots+5\lambda_5=5}\prod_{j=1}^m\binom{7-j^2-\sum_{i=j+1}^5\lambda_i}{\lambda_j}.$$
In the last sum each summand is corresponded to each of the
seven solutions of the above equation. Moreover, only those
summands are non-zero that are corresponded to the first,
second and fourth solutions of above equation. Thus,
$$|C_5(A)|=\binom{6}{5}+\binom{5}{3}\cdot\binom{3}{1}+\binom{4}{1}\cdot\binom{3}{3}=6+30+12=48.$$
\end{example}

Consider the case of a linear multiset.

\begin{theorem}\label{theo.lin}
Suppose $A$ is a linear multiset with the primary specification
$$k(A)=\{pi+q:i=1,\ldots,n\},$$
where $1\leqslant p+q$, $q\in Z$, $p\in N_0$; then we have
\begin{equation}\label{p.lin}
|C_m(A)|=\sum_{\lambda\in \Lambda^m(A)}\prod_{j=1}^s
\binom{\left\lfloor n-\max\left(1,\frac{j-q}{p}\right)\right\rfloor+1-\sum_{i=j+1}^s\lambda_i}{\lambda_j}, \,\, p\neq 0,
\end{equation}
where $s=\min(m,pn+q)$.

If $m\leqslant n$, then equality (\ref{p.lin}) have the form
\begin{equation}\label{lin.m<n}
|C_m(A)|=\sum_{\lambda_1+\ldots+s\lambda_s=m}\prod_{j=1}^s \binom{\left\lfloor
n-\max\left(1,\frac{j-q}{p}\right)\right\rfloor+1-\sum_{i=j+1}^s\lambda_i}{\lambda_j},\,\, p\neq 0.
\end{equation}
\end{theorem}

\begin{proof}
Consider first the case when $p\neq 0$. Since the linear
function $f(i)=pi+q$ satisfies the conditions of Theorem
\ref{theo.f(x)}, we have
\begin{equation}\label{-k.lin}
\overline{k_j}=\left\lfloor n-\max\left(1,\frac{j-q}{p}\right)\right\rfloor+1.
\end{equation}
Hence equality (\ref{p.lin}) is valid.

In addition, suppose that $m\leqslant n$; then, using equality
(\ref{-k.lin}) and inequality $px+q\geqslant x$, $x\in[1,n]$,
from Theorem \ref{theo.C_m(A)} it follows equality
(\ref{lin.m<n}).
\end{proof}

\begin{theorem}\label{theo.cons}
The number of $m$-submultisets of the constant multiset
$$A=\{a_1^q,a_2^q,\ldots,a_n^q\}$$
can be obtained by the following formulas.

1)\begin{equation}\label{q.cons}|C_m(A)|= \sum_{\lambda\in
\Lambda^m(A)}\frac{n!}{\lambda_1!\cdot\ldots\cdot\lambda_r!(n-\lambda_1-\ldots-\lambda_r)!},
\end{equation}
where $r=\min(m,q).$

2) If $m\leqslant q,$ then
\begin{equation}\label{m<q}|C_m(A)|=
\sum_{\lambda_1+2\lambda_2+\ldots+m\lambda_m=m}\frac{n!}{\lambda_1!\cdot\ldots\cdot\lambda_m!(n-\lambda_1-\ldots-\lambda_m)!}.
\end{equation}

3) If $m\leqslant n,$ then
\begin{equation}\label{q.cons.m<n}
|C_m(A)|=\sum_{\lambda_1+\ldots+s\lambda_s=m}\frac{n!}{\lambda_1!\cdot\ldots\cdot\lambda_s!(n-\lambda_1-\ldots-\lambda_s)!},
\end{equation}
where $s=\min(m,q).$
\end{theorem}

\begin{proof}
1) In the case of a constant multiset we have $s=\min(m,q)$ and
$\overline{k_j}=n$, $j=1,2,\ldots,s.$ Therefore,
$$|C_m(A)|=\sum_{\lambda\in\Lambda^m(A)}\prod_{j=1}^s\binom{n-\sum_{i=j+1}^s\lambda_i}{\lambda_j}=
\sum_{\lambda\in\Lambda^m(A)}\frac{n!}{\lambda_1!\cdot\ldots\cdot\lambda_s!(n-\lambda_1-\ldots-\lambda_s)!}.$$

2) If $m\leqslant q$, then the set $\Lambda^m(A)$ coincides
with the set of all solutions of the equation
$$\lambda_1+2\lambda_2+\ldots+m\lambda_m=m.$$
Therefore formula (\ref{q.cons}) has the form (\ref{m<q}).

3) If $m\leqslant n,$ then inequalities (\ref{sumilambda<kj})
hold for all $j=1,\ldots,s$ and equality (\ref{q.cons.m<n}) is
valid.
\end{proof}

\begin{remark}From Theorem \ref{theo.cons} (see item 2)) and equality (\ref{infty,infty}) it follows that
$$
\sum_{\lambda_1+2\lambda_2+\ldots+m\lambda_m=m}\frac{n!}{\lambda_1!\cdot\ldots\cdot\lambda_m!(n-\lambda_1-\ldots-\lambda_m)!}=\binom{n+m-1}{m}.$$

Notice that the left side of this identity consists only of
those summands such that
$n-(\lambda_1+\ldots+\lambda_m)\geqslant 0.$
\end{remark}

\section{Generatix method}

A generatix is the function
$$f(t)=\prod_{i=1}^n\sum_{j=0}^{k_i}t^j=\sum_{i=0}^{k_1+\ldots+k_n}|C_i(A)|t^i$$
for the calculation of the number of $m$-submultisets of the
multiset
$$A=\{a_1^{k_1},\ldots,a_n^{k_n}\}.$$
Therefore, after $m$-fold differentiation of this function we
obtain the equality
$$|C_m(A)|=\frac{1}{m!}\cdot\frac{d^mf(t)}{dt^m}\left|_{t=0}.\right.$$
We have
$$\frac{d^mf(t)}{dt^m}=\sum_{r_1+\ldots+r_n=m}\frac{m!}{r_1!\cdot\ldots\cdot r_n!}
\frac{d^{r_1}g_1(t)}{dt^{r_1}}\cdot\ldots\cdot\frac{d^{r_n}g_n(t)}{dt^{r_n}},$$
where $g_i(t)=1+t+t^2+\ldots+t^{k_i}.$ Since
$$\frac{d^{r_i}(g_i(t))}{dt^{r_i}}=r_i^{\underline{r_i}}+(r_i+1)^{\underline{r_i}}t+\ldots+k_i^{\underline{r_i}}t^{k_i-r_i}$$
and
$$\frac{d^{r_i}(g_i(t))}{dt^{r_i}}\left|_{t=0}\right.=r_i!,$$
we obtain
$$|C_m(A)|=\frac{1}{m!}\cdot\frac{d^mf(t)}{dt^m}\left|_{t=0}\right.=\frac{1}{m!}\sum_{r_1+\ldots+r_n=m}\frac{m!}{r_1!\cdot\ldots\cdot
r_n!}\cdot r_1!\cdot\ldots\cdot
r_n!=
\sum_{r_1+\ldots+r_n=m}1,$$
where $0\leqslant r_i\leqslant k_i.$

Thus, we have the next theorem.

\begin{theorem}\label{theo.|C^m(A)|=sum 1}
The number of $m$-submultisets of the multiset
$$A=\{a_1^{k_1},\ldots,a_n^{k_n}\}$$
is equal to
\begin{equation}\label{|C^m(A)|=sum1}|C_m(A)|=\sum_{\begin{subarray}{l}r_1+r_2+\ldots+r_n=m\\
0\leqslant r_i\leqslant k_i, \,i=1,\ldots,n\end{subarray}}1.
\end{equation}
\end{theorem}

Let us use Theorem \ref{theo.|C^m(A)|=sum 1} to determine the
formula for the calculation of the number of $m$-submultisets
of the constant multiset
$$A=\{a_1^q,a_2^q,\ldots,a_n^q\}.$$
First note that if the solution $(s_1,s_2,\ldots,s_n)$ of the
equation
\begin{equation}\label{r=m}
r_1+r_2+\ldots+r_n=m
\end{equation}
satisfies the inequalities
$$0\leqslant s_i\leqslant q,$$
then an arbitrary permutation of the components of this
solution leads to a new solution of this equation. Therefore we
need to find all disordered solutions of equation (\ref{r=m}),
i. e., such solutions $(r_1,r_2,\ldots,r_n)$ that satisfy the
inequalities $r_1\geqslant r_2\geqslant\ldots\geqslant
r_n\geqslant 0$ and we need to count the number of permutations
of the components of each solution. Suppose among the
components of solution $(r_1,r_2,\ldots,r_n)$ are $\lambda_0$
zeros, $\lambda_1$ ones, and so on; then all disordered
solutions of equation (\ref{r=m}) can be counted using the
system of equations
$$\begin{cases}0\lambda_0+1\lambda_1+\ldots+q\lambda_q=m,\\
\lambda_0+\lambda_1+\ldots+\lambda_q=n. \end{cases}$$
Therefore,
$$C_m(A)=\sum_{\begin{subarray}\, 0\cdot\lambda_0+1\lambda_1+\ldots+q\lambda_q=m\\ \lambda_0+\lambda_1+\ldots+\lambda_q=n\end{subarray}}
\frac{n!}{\lambda_0!\lambda_1!\cdot\ldots\cdot\lambda_q!}.$$

Thus the next theorem is valid.

\begin{theorem}\label{theo.cons.q}
The number of $m$-submultisets of the constant multiset
$$A=\{a_1^q,a_2^q,\ldots,a_n^q\}$$ is equal to
\begin{equation}\label{cons.q}C_m(A)=\sum_{\begin{subarray}\,0\cdot\lambda_0+1\lambda_1+\ldots+q\lambda_q=m\\ \lambda_0+\lambda_1+\ldots+\lambda_q=n\end{subarray}}
\frac{n!}{\lambda_0!\lambda_1!\cdot\ldots\cdot\lambda_q!}.
\end{equation}
\end{theorem}

\begin{remark}
If in Theorem (\ref{theo.cons.q}) $m\leqslant q,$ then, using
equality (\ref{infty,infty}) (see page \pageref{infty,infty}),
we obtain the following combinative identity
$$\sum_{\begin{subarray}\,0\lambda_0+1\lambda_1+\ldots+m\lambda_m=m\\ \lambda_0+\lambda_1+\ldots+\lambda_m=n\end{subarray}}
\frac{n!}{\lambda_0!\lambda_1!\cdot\ldots\cdot\lambda_m!}=\binom{m+n-1}{m}.$$
\end{remark}

\begin{example}
Let us find the number of those $m$-submultisets of the
multiset $A=\{x_1^{\infty},x_2^{\infty},\ldots,x_n^{\infty}\}$
such that they contain each element of basis $[A]$ (see
\cite{Riordan1}) of the multiset $A$. To find them we use the
generatrix
$$W(t)=\left(\sum_{i=1}^{\infty}t^i\right)^n=t^n(1-t)^{-n}=t^n\sum_{i=0}^{\infty}\frac{n^{\overline{i}}}{i!}t^i=
\sum_{i=0}^{\infty}\frac{n^{\overline{i}}}{i!}t^{n+i}.
$$
Put $n+i=m$; then
$$W(t)=\sum_{m=n}^{\infty}\frac{n^{\overline{m-n}}}{(m-n)!}t^m=\sum_{m=n}^{\infty}\binom{m-1}{n-1}t^m.$$
\end{example}

We shall consider one more class of multisets with primary
specification (\ref{specyf_step}), i. e.,
$$A=\left\{a_1^{2^{l_1}-1},\ldots,a_n^{2^{l_n}-1}\right\}\!,
\; l_1\leqslant l_2\leqslant\ldots\leqslant l_n
$$
such that their number of $m$-submultisets is calculated
relatively easily. As shown in \cite{Riordan1}, the generatrix
of the number of $m$-submultisets of such multisets has the
form
$$W(t)=\prod_{i=1}^n\sum_{j=0}^{2^{l_i}-1}t^j.$$

However
$$1+t+\ldots+t^{{2^l}-1}=(1+t)(1+t^2)(1+t^4)\cdot\ldots\cdot(1+t^{2^{l-1}})$$
whence, using the designation
$$t^{2^i}=x_{i+1},\; i=0,\ldots,l-1,$$
we get
$$W(t)=(1+x_1)^{m_1}\cdot\ldots\cdot(1+x_{l_n})^{m_n}.$$
Obviously, the number $|C_m(A)|$ is equal to the sum of
coefficients of the monomials
$$K(\lambda_1,\ldots,\lambda_{l_n})x_1^{\lambda_1}\cdot\ldots\cdot
x_{l_n}^{\lambda_{l_n}}
$$
with $l_n$ variables such that their indices $\lambda_1$,
$\lambda_2$, \ldots, $\lambda_{l_n}$ are the components of
solutions of the equation
$$\lambda_1+2\lambda_2+2^2\lambda_3+\ldots+2^{l_n-1}\lambda_{l_n}=m$$
and these indices satisfy the inequalities
$$\lambda_i\leqslant \overline{k_{2^{i-1}}}, \;i=1,\ldots,l_n,$$
where $(\overline{k_1},\ldots,\overline{k_{2^{l_n}-1}})$ is the
specification of the multiset $\overline{A}$ that is adjoint to
the multiset $A.$ Therefore,
$$|C_m(A)|=\sum_{\begin{subarray}\lambda_1+2\lambda_{2}+\ldots+2^{l_n-1}\lambda_{l_n}=m\\
\lambda_i\leqslant k_{2^{i-1}},
\; i=1,\ldots,l_n\end{subarray}}\prod_{i=1}^{l_n}\binom{\overline{k_{2^{i-1}}}}{\lambda_i}.$$

Thus the next theorem is valid.

\begin{theorem}\label{theo.step}
The number of $m$-submultisets of the multiset
$$A=\{a_1^{2^{l_1}-1},\ldots,a_n^{2^{l_n}-1}\}, \;l_1\leqslant
l_2\leqslant\ldots\leqslant l_n
$$
is equal to
\begin{equation}\label{|C^m(A)|.step}
|C^m(A)|=\sum_{\begin{subarray} \lambda_1+2\lambda_{2}+\ldots+2^{l_n-1}\lambda_{l_n}=m\\
\lambda_i\leqslant \overline{k_{2^{i-1}}}, \,i=1,\ldots,l_n\end{subarray}}\prod_{i=1}^{l_n}\binom{\overline{k_{2^{i-1}}}}{\lambda_i}.
\end{equation}
\end{theorem}

\begin{example}
Suppose we have the multiset
$A=\{a_1^3,a_2^7,a_3^{15},a_4^{31}\}.$ We seek the primary
specification of the adjoint multiset $\overline{A}:$
$$\overline{k}=(4,4,4,3,3,3,3,2,2,2,2,2,2,2,2,1,1,1,1,1,1,1,1,1,1,1,1,1,1,1,1).$$
Therefore,
$$\overline{k_1}=4, \;\overline{k_2}=4,
\;\overline{k_4}=3, \;\overline{k_8}=2,
\;\overline{k_{16}}=1.
$$
The equation
$$\lambda_1+2\lambda_2+4\lambda_3+8\lambda_4+16\lambda_5=21$$
have $60$ solutions. But only $13$ solutions satisfy the
inequalities
$$\lambda_1\leqslant\overline{k_1}, \,\lambda_2\leqslant\overline{k_2}, \,
\lambda_3\leqslant\overline{k_4},\,\lambda_4\leqslant\overline{k_8}, \,\lambda_5\leqslant\overline{k_{16}}.
$$
List of these solutions:
$$(3,3,3,0,0),\;(3,3,1,1,0),\;(3,1,2,1,0),\;(3,1,0,2,0),\;(3,1,0,0,1),$$
$$(1,4,3,0,0),\;(1,4,1,1,0),\;(1,2,2,1,0),\;(1,2,0,2,1),\;(1,2,0,0,1),$$
$$(1,0,3,1,0),\;(1,0,1,2,0),\;(1,0,1,0,1).$$
Thus, we have
$$|C^{21}(A)|=\binom{4}{3}\cdot\binom{4}{3}\cdot\binom{3}{3}+\binom{4}{3}\cdot\binom{4}{3}\cdot\binom{3}{1}\cdot\binom{2}{1}+
\binom{4}{3}\cdot\binom{4}{1}\cdot\binom{3}{2}\cdot\binom{2}{1}+$$$$\binom{4}{3}\cdot\binom{4}{1}\cdot\binom{2}{2}+
\binom{4}{3}\cdot\binom{4}{1}\cdot\binom{1}{1}+\binom{4}{1}\cdot\binom{4}{4}\cdot\binom{3}{3}+
\binom{4}{1}\cdot\binom{4}{4}\cdot\binom{3}{1}\cdot\binom{2}{1}+$$$$\binom{4}{1}\cdot\binom{4}{2}\cdot\binom{3}{2}\cdot\binom{2}{1}+
\binom{4}{1}\cdot\binom{4}{2}\cdot\binom{2}{2}\cdot\binom{1}{1}+\binom{4}{1}\cdot\binom{4}{2}\cdot\binom{1}{1}+
\binom{4}{1}\cdot\binom{3}{3}\cdot\binom{2}{1}+$$$$\binom{4}{1}\cdot\binom{3}{1}\cdot\binom{2}{2}+
\binom{4}{1}\cdot\binom{3}{1}\cdot\binom{1}{1}=492.$$
\end{example}

\section{Algorithm for calculation of $m$-submultisets of an arbitrary
multiset}

Let us construct a recursive algorithm for calculation of the
number of $m$-submultisets of the multiset
$$A=\{a_1^{k_1},a_2^{k_2},\ldots,a_n^{k_n}\}.$$

We use for the number $C_m(A)$ the notation from
\cite{Grin-Kleit}. Then we have
$$\prod_{i=1}^n(1+t+t^2+\ldots+t^{k_i})=\sum_{i=0}^r\binom{k_1,k_2,\ldots,k_n}{i}t^i,$$
where $r=\sum_{i=1}^nk_i.$ If the coefficients
$$A(i)=\binom{k_1,k_2,\ldots,k_{l-1}}{i},\;\; i=0,\ldots,s
$$
of the polynomial
$$\sum_{i=0}^sA(i)t^i=\prod_{i=1}^{l-1}(1+t+\ldots+t^{k_i}),\quad s=\sum_{i=1}^{l-1}k_i
$$
are known, then the coefficients
$$B(j)=\binom{k_1,k_2,\ldots,k_l}{j}, \;j=0,1,\ldots,s+k_l$$
of the polynomial
$$\sum_{j=0}^{s+k_l}B(j)t^j=(1+t+\ldots+t^{k_l})\cdot \sum_{i=0}^sA(i)t^i
$$
are obtained by summing $k_l+1$ last elements of the row
$$\underbrace{0\ldots 0}_{k_l}A(0)A(1)\ldots A(j).
$$
More exactly, we have
\begin{equation}\label{l-l+1}\binom{k_1,k_2,\ldots,k_l}{j}=\sum_{q=j-k_l}^j\binom{k_1,k_2,\ldots,k_{l-1}}{q},
\; j=0,\ldots,s+k_l,
\end{equation}
where
$$\binom{k_1,k_2,\ldots,k_{l-1}}{q}=0$$
if $q<0$ or $q>s.$

The calculation process is convenient to design in the form of
a generalized Pascal triangle.

If the first element of the multiset $A$ has multiplicity
$k_1,$ then the calculation is begun from the zero row of the
table
$$\underbrace{0\ldots 0}_{k_1}1\underbrace{0\ldots 0}_{k_1}.$$
The first row of this table is obtained with the help of the
zero row by using relation (\ref{l-l+1}). Then the first row
has the form
$$\underbrace{0\ldots 0}_{k_2}\underbrace{1\ldots
1}_{k_1+1}\underbrace{0\ldots 0}_{k_2},
$$
where $k_2$ is the multiplicity of the second element of
multiset $A.$

Continuing the calculation process to the $n$th line inclusive,
we obtain the required numbers
$$C_i(A), \;i=0,\ldots,r+1, \;r=\sum_{i=1}^nk_i.$$

\begin{example}
Suppose we have the multiset
$$A=\{a_1^4,a_2^3,a_3^3,a_4^1\};$$
then, using the above algorithm, we obtain the table
$$\arraycolsep=2pt
\begin{array}{cccccccccccccccc}
   &  &  &  & C_0(A)\mspace{-2mu} & C_1(A)\mspace{-2mu} & C_2(A)\mspace{-2mu} & C_3(A)\mspace{-2mu}
   & C_4(A)\mspace{-2mu} & C_5(A)\mspace{-2mu} & C_6(A)\mspace{-2mu} & C_7(A)\mspace{-2mu}
   & C_8(A)\mspace{-2mu} & C_9(A)\mspace{-2mu} & C_{10}(A)\mspace{-2mu} &C_{11}(A)\\
  0 & 0 & 0 & 0 & 1 & 0 & 0 & 0 & 0 &  &  &  &  &  &  &  \\
   & 0 & 0 & 0 & 1 & 1 & 1 & 1 & 1 & 0 & 0 & 0 &  &  &  &  \\
   & 0 & 0 & 0 & 1 & 2 & 3 & 4 & 4 & 3 & 2 & 1 & 0 & 0 & 0 &   \\
   &  &  & 0 & 1 & 3 & 6 & 10 & 13 & 14 & 13 & 10 & 6 & 3 & 1 & 0   \\
   &  &  &  & 1 & 4 & 9 & 16 & 23 & 27 & 27 & 23 & 16 & 9 & 4 & 1  \\
\end{array}
$$

The results are written in the last row of this table:
$C_0(A)=1$, $C_1(A)=4$, $C_2(A)=9$, $C_3(A)=16$, \ldots{} At
the same time the equality
$$\sum_{i=0}^{11}C_i(A)=(4+1)(3+1)(3+1)(1+1)=160$$
is fulfilled.
\end{example}

The Pascal triangle for the set
$A=\{a_1^1,a_2^1,a_3^1,a_4^1\}$, according to the above
algorithm, has the form
$$
\begin{array}{cccccc}
    & C_0(A) & C_1(A) & C_2(A) & C_3(A) & C_4(A) \\
  0 & 1 & 0 & & &    \\
    0 & 1 & 1 & 0 &  &  \\
    0 & 1 & 2 & 1 & 0&   \\
 0 & 1 & 3 & 3 & 1 & 0  \\
   & 1 & 4 & 6 & 4 & 1   \\
\end{array}
$$
In addition, we have the relation
$$\sum_{i=0}^{4}C_i(A)=(1+1)(1+1)(1+1)(1+1)=16.$$

\begin{remark}
If the multiplicities of the multiset elements are large, then
it is convenient to use the relations
$$\binom{k_1,k_2,\ldots,k_l}{j}=\binom{k_1,k_2,\ldots,k_l}{j-1}+\binom{k_1,k_2,\ldots,k_{l-1}}{j}-
\binom{k_1,k_2,\ldots,k_{l-1}}{j-k_l-1},
$$
$$
j=0,\ldots,\sum_{i=1}^lk_i,
$$
which follow from relations (\ref{l-l+1}). This can
significantly reduce the number of operations.
\end{remark}

\begin{remark}
Since any $k$-submultiset $B$ of the multiset $A$ uniquely
corresponds to $(|A|-k)$-submultiset $A-B$ of this multiset, we
have
$$
\binom{k_1,k_2,\ldots,k_l}{j}=\binom{k_1,k_2,\ldots,k_l}{s-j}, \;j=0,\ldots,s,
$$
where $s=\sum_{i=1}^lk_i,$ i. e., the numbers that are
equidistant from the ends of each row of the table are equal to
each other. Thus, if $m>\lfloor\frac{s}{2}\rfloor,$ then
instead of calculating  $|C_m(A)|$ it is more convenient to
calculate $|C_{s-m}(A)|.$
\end{remark}

\section{$m$-permutations of the multiset elements}

\begin{definition}\label{def.compoz}
{\sloppy The set of all ordered $m$-samples of elements of the
multiset $A=\{a_1^{k_1},\ldots,a_n^{k_n}\}$ is called the set
of $m$-permutations on this multiset. By $P_m(A)$ we denote
this set.

}
\end{definition}

The following statement is well known.

\begin{proposition}
\label{stat.polinom.formula} The number of all permutations of
elements of the multiset $A$ is equal to
\begin{equation*}\label{polinom.formula}
    |P_{|A|}(A)|=\frac{(k_{1}+k_{2}+
\ldots+k_{n})!}{k_{1}!k_{2}!\cdot\ldots\cdot k_{n}!}.
\end{equation*}
\end{proposition}

To determine the number of all $m$-permutations of the multiset
$A=\{a_1^{k_1},\ldots,a_n^{k_n}\}$ we use the theorem (see
Theorem \ref{theo.C_m(A)} on page \pageref{theo.C_m(A)}) about
the number of all $m$-submultisets ($m$-combinations) of this
multiset.

In this theorem it was found that the number of all
$m$-combinations of the multiset $A$ is equal to
$$|C_m(A)|=\sum_{\lambda\in
\Lambda_m(A)}|C^{\lambda}_m(A)|,
$$
where
$$C^{\lambda}_m(A)=\{B\in
C_m(A):B^{''}=\{\lambda_1,\ldots,\lambda_n\}\}.
$$
Obviously,
\begin{equation}\label{|P^m(A)|}
|P_m(A)|=\sum_{\lambda\in \Lambda_m(A)}|P_{|B|}(B)|\cdot |C^{\lambda}_m(A)|
\end{equation}
but
$$|P_{|B|}(B)|=\frac{m!}{1!^{\lambda_1}2!^{\lambda_2}\cdot\ldots\cdot
s!^{\lambda_s}}
$$
whence equality (\ref{|P^m(A)|}) leads to the following
theorem.

\begin{theorem}\label{theo.|P^m(A)|}
The number of all $m$-permutations of the multiset
$A=\{a_1^{k_1},\ldots,a_n^{k_n}\}$ is equal to
\begin{equation}\label{|P^m(A)|!}
|P_m(A)|=\sum_{\lambda\in \Lambda_m(A)}\frac{m!}{1!^{\lambda_1}2!^{\lambda_2}\cdot\ldots\cdot
s!^{\lambda_s}}\prod_{j=1}^s\binom{\overline{k_j}-\sum_{i=j+1}^s\lambda_i}{\lambda_j},
\end{equation}
where $\Lambda_m(A)$ is the set of those solutions of the
equation
\begin{equation}\label{sumilambda=mper}
\sum_{i=1}^si\lambda_i=m
\end{equation}
that satisfy the inequalities
\begin{equation}\label{sumilambda<kjper}
\sum_{i=j}^s\lambda_i\leqslant \overline{k_j};\; j=1,\ldots,s,
\end{equation}
$$s=\min(m,r),\; r=\max\{k_i\},\; i=1,\ldots,n,$$
$\overline{k_j}$ is the $j$th element of specification
(\ref{-k=k}), which is adjoint to the primary specification of
the multiset $A.$
\end{theorem}

The number of solutions of the equation
$\sum_{i=1}^si\lambda_i=m$ increases with increasing $m$ and
$s$. For example, already at $m=s=20$ this equation has 627
solutions. Consequently formula (\ref{|P^m(A)|!}) is not always
convenient for practical use because it requires large amounts
of computation.

We construct an algorithm for calculating $m$-permutations of
elements of the multiset $A=\{a_1^{k_1},\ldots,a_n^{k_n}\}$
such that in many cases this algorithm eliminates these
shortcomings.

Let
$$|P_m(A)|=\left[\begin{array}{c}
                         k_1 k_2 \cdots  k_n \\
                          m
                       \end{array}
\right]\!.
$$

In particular, if $k_1=k_2=\ldots=k_n=1,$ then the multiset
coincides with its basis and this multiset is an ordinary set,
i. e.,
$$\left[\begin{array}{c}
          \underbrace{1\, 1 \ldots\, 1}_{n} \\
          m
        \end{array}
\right]=\frac{n!}{(n-m)!}, \; 0\leqslant m\leqslant n.$$

If $k_1=n$, $k_2=k_3=\ldots=k_n=0,$ then
$$\left[\begin{array}{c}
                                                     n \\
                                                     m
                                                   \end{array}
\right]=1, \;0\leqslant m\leqslant n.
$$

In the case, where the multiset $A$ has specification
(\ref{specyf_infty}) (see page \pageref{specyf_infty}), we have
the obvious equality
$$\left[\begin{array}{c}
          \underbrace{\infty\, \infty\, \cdots\, \infty}_{n} \\
          m
        \end{array}
\right]=n^m.$$

\begin{theorem}\label{theo.algor.permut}
For any $r=2,3,\ldots,n$ the equality
$$\left[\begin{array}{c}
                                               k_1\,k_2\,\ldots\,k_r \\
                                               i
                                             \end{array}
\right]=$$\begin{equation}\label{eq.algor.permut}
=\begin{cases}\sum_{j=0}^{\min(i,k_1+\ldots+k_{r-1})}\binom{i}{j}\left[\begin{array}{c}
                                               k_1\,k_2\,\ldots\,k_{r-1} \\
                                               i
                                             \end{array}
\right]\!,\; i\leqslant k_r,\\
\sum_{j=i-k_r}^{\min(i,k_1+\ldots+k_{r-1})}\binom{i}{j}\left[\begin{array}{c}
                                               k_1\,k_2\,\ldots\,k_{r-1} \\
                                               i
                                             \end{array}
\right]\!,\; k_r<i\leqslant k_1+\ldots+k_r \end{cases}
\end{equation}
is fulfilled, where $0\leqslant i\leqslant |A|.$
\end{theorem}
\begin{proof}
The generatrix for the number of permutations
$$\left[\begin{array}{c}
k_1\,k_2\,\ldots\,k_r \\
i
\end{array}
\right]
$$
of elements of the multiset $A=\{a_1^{k_1},\ldots,a_n^{k_n}\}$
has the form
$$\prod_{i=1}^n\sum_{j=0}^{k_i}\frac{t^j}{j!}=\sum_{i=0}^{k_1+\ldots+k_n}\left[\begin{array}{c}
k_1\,k_2\,\ldots\,k_n \\
i
\end{array}
\right]\frac{t^i}{i!}.
$$
Hence,
$$\left(\sum_{j=0}^{k_1+\ldots+k_{r-1}}\left[\begin{array}{c}
k_1\,k_2\,\ldots\,k_{r-1} \\
j
\end{array}
\right]\frac{t^j}{j!}\right)\cdot\sum_{s=0}^{k_r}\frac{t^s}{s!}=\sum_{i=0}^{k_1+\ldots+k_r}\left[\begin{array}{c}
k_1\,k_2\,\ldots\,k_r \\
i
\end{array}
\right]\frac{t^i}{i!}.
$$
Since
$$\left(\sum_{j=0}^{k_1+\ldots+k_{r-1}}\left[\begin{array}{c}
k_1\,k_2\,\ldots\,k_{r-1} \\
j
\end{array}
\right]\frac{t^j}{j!}\right)\sum_{s=0}^{k_r}\frac{t^s}{s!}=$$
$$=\sum_{i=0}^{k_1+\ldots+k_r}\left(\sum_{j+s=i}\left[\begin{array}{c}
k_1\,k_2\,\ldots\,k_{r-1} \\
j
\end{array}
\right]\frac{t^i}{j!s!}\right)\!,
$$
we have
$$\left[\begin{array}{c}
k_1\,k_2\,\ldots\,k_r \\
i
\end{array}
\right]=\sum_{j+s=i}\frac{i!}{j!s!}\left[\begin{array}{c}
k_1\,k_2\,\ldots\,k_{r-1} \\
j
\end{array}
\right]=\sum_{j+s=i}\binom{i}{j}\left[\begin{array}{c}
k_1\,k_2\,\ldots\,k_{r-1} \\
j
\end{array}
\right]\!.
$$

For both expressions $\binom{i}{j}$ and $\left[\begin{array}{c}
k_1,\ldots,k_{r-1} \\
j
\end{array}
\right]$ in the last sum to have meaning, it is necessary to
have the inequalities $j\leqslant i$ and $j\leqslant
k_1+\ldots+k_{r-1}$, i. e., the inequality $j\leqslant
\min(i,k_1+\ldots+k_{r-1})$ is valid. If $i\leqslant k_r,$ then
from the inequality $0\leqslant s\leqslant k_r$ it follows that
the smallest value of the index $j$ under the restriction
$j+s=i$ is $j=0.$ If $i>k_r,$ then the smallest value of the
index $j$ is $j=i-k_r.$ This completes the proof.
\end{proof}

Recurrence equality (\ref{eq.algor.permut}) can be used to
calculate the number of all $m$-permutations of the multiset
$A=\{a_1^{k_1},\ldots,a_n^{k_n}\},$ where $m=0,\ldots,|A|.$

For this purpose

1. Write the row of $k_1+1$ ones, which are numbers of
$i$-permutations $\left[\begin{array}{c}
k_1 \\
i
\end{array}
\right]$ on the multiset $A=\{a_1^{k_1}\}, i=0,\ldots,k_1.$
This row is called the basic row.

2. Under the basic row we construct a table with $k_1+1$
columns and $k_1+k_2+1$ rows. We number rows of the table from
top to bottom by numbers from $0$ to $k_1+k_2.$

3. In the $i$th row of the table we write the first $k_1+1$
elements of the $i$th row of the Pascal triangle. If the $i$th
row of the Pascal triangle contains the less than $k_1+1$
elements, then we add the required number of zeros.

4. In the lower left corner of the table we replace the written
numbers by zeros so that the zeros form a right isosceles
triangle with the leg $k_1.$

5. We calculate the sum of the products of elements for the
$i$th $(i=0,\ldots,k_1+k_2)$ row of the table and the
corresponding elements of the basic row. The resulting number
of permutations
$$\left[\begin{array}{c}
k_1,k_2 \\
i
\end{array}
\right]\!, \;i=0,\ldots,k_1+k_2
$$
is added to the $i$th row on the right.

6. If the number of rows of the last table is greater than the
cardinality of the multiset, then the calculation is completed
and the result of the algorithm is the column of numbers such
that these numbers were added to the table on the right.
Otherwise, we transpose the column of numbers that were added
to the table on the right, consider this as the base row of the
new table, the parameters of the table are increased by the
value of the multiplicity of the next element of the multiset,
and then we go to item 2.

Thus, if the multiset $A$ has the cardinality basis $n,$ then
the execution of the algorithm requires the construction of the
$(n-1)$th table.

\begin{example}
Find the number of all $m$-permutations of the multiset
$A=\{a_1^2,a_2^4,a_3^5\}, \,m=0,1,\ldots,11.$

For this purpose we build the following tables:
$$\begin{array}{ccccccc}
     &  & 1 & 1 & 1 &  &  \\
    \hline
    0 & \vline & 1 & 0 & 0 & \vline & 1 \\
    1 & \vline & 1 & 1 & 0 & \vline & 2 \\
    2 & \vline & 1 & 2 & 1 & \vline & 4 \\
    3 & \vline & 1 & 3 & 3 & \vline & 7 \\
    4 & \vline & 1 & 4 & 6 & \vline & 11 \\
    5 & \vline & 0 & 5 & 10 & \vline & 15 \\
    6 & \vline & 0 & 0 & 15 & \vline & 15
  \end{array}
$$

$$\begin{array}{ccccccccccc}
     &  & 1 & 2 & 4 & 7 & 11 & 15 & 15 &  &  \\
\hline
    0 & \vline & 1 & 0 & 0 & 0 & 0 & 0 & 0 & \vline & 1 \\
    1 & \vline & 1 & 1 & 0 & 0 & 0 & 0 & 0 & \vline & 3 \\
    2 & \vline & 1 & 2 & 1 & 0 & 0 & 0 & 0 & \vline & 9 \\
    3 & \vline & 1 & 3 & 3 & 1 & 0 & 0 & 0 & \vline & 26 \\
    4 & \vline & 1 & 4 & 6 & 4 & 1 & 0 & 0 & \vline & 72 \\
    5 & \vline & 1 & 5 & 10 & 10 & 5 & 1 & 0 & \vline & 191 \\
    6 & \vline & 0 & 6 & 15 & 20 & 15 & 6 & 1 & \vline & 482 \\
    7 & \vline & 0 & 0 & 21 & 35 & 35 & 21 & 7 & \vline & 1134 \\
    8 & \vline & 0 & 0 & 0 & 56 & 70 & 56 & 28 & \vline & 2422 \\
    9 & \vline & 0 & 0 & 0 & 0 & 126 & 126 & 34 & \vline & 4536 \\
    10 & \vline & 0 & 0 & 0 & 0 & 0 & 252 & 210 & \vline & 6930 \\
    11 & \vline & 0 & 0 & 0 & 0 & 0 & 0 & 462 & \vline & 6930
  \end{array}
$$
Therefore,
$$P^0(A)=1,\, P^1(A)=3,\, P^2(A)=9,\, P^3(A)=26,\,
P^4(A)=72,\, P^5(A)=191,$$
$$ P^6(A)=482,\, P^7(A)=1134,\,
P^8(A)=2422,\, P^9(A)=4536,\, P^{10}(A)=6930,$$
$$ P^{11}(A)=6930.$$
\end{example}

This algorithm is effective for multisets of relatively large
cardinality but with a small base. For example, to calculate
the number of $20$-permutations on the multiset
$A=\{a_1^3,a_2^9,a_3^{13}\}$ this algorithm requires the
construction of two tables of sizes $4\times 13$ and $13\times
26$ accordingly and the calculation by the formula requires the
analysis of the set of 627 solutions of equation
(\ref{sumilambda=mper}) and significant calculations.


\begin{example}
For the multiset
$$A=\{a_1^1,a_2^1,a_3^1,a_4^2,a_5^2,a_6^2,a_7^2,a_8^3,a_9^3,a_{10}^5\}$$
we have
$$P^0(A)=1,\, P^1(A)=10,\, P^2(A)=97,\, P^3(A)=912,\, P^4(A)=8299,$$
$$ P^5(A)=72946,\, P^6(A)=617874,\, P^7(A)=5029948,\, P^8(A)=39237380,$$
$$ P^9(A)=292327224,\, P^{10}(A)=2072330400,\, P^{11}(A)=13920355680,$$
$$ P^{12}(A)=88179787080,\, P^{13}(A)=523856052720,\, P^{14}(A)=2899520704080,$$
$$ P^{15}(A)=14831963546400,\, P^{16}(A)=6938695764000,$$
$$ P^{17}(A)=292608485769600,\, P^{18}(A)=1088829613872000,$$
$$ P^{19}(A)=3456466684070400,\, P^{20}(A)=8834757003072000, $$
$$P^{21}(A)=162615846032640000,\, P^{22}(A)=162615846032640000.$$
\end{example}



\begin{thebibliography}{99}
\bibitem{Aigner}M. Aigner. \emph{Combinatorial theory}.
    Springer-Verlag, 1979.


\bibitem{Babai}L. Babai and P. Frankl. \emph{Linear algebra
    methods in combinatorics with applications to geometry and
    computer science}. Preliminary Version 2. Department of Computer Science the University of Chicago,
    1992.

\bibitem{Cartier_Foata}P. Cartier and D. Foata.
    \emph{Probl\`{e}mes combinatoires de commutation et
    r\'{e}arrangements}, volume 85 of
    \emph{Lecture Notes in Mathematics}. Springer-Verlag, 1969.

\bibitem{David_Barton}F. N. David and D. E. Barton.
    \emph{Combinatorial Chance}. Griffin, 1962.

\bibitem{Rota}P. Doubilet, G.-C. Rota, and R. Stanley. On
    the foundations of combinatorial theory. VI. The idea of
    generating function. In \emph{Proceedings of the Sixth Berkely
    Symposium on Mathematical Statistics and Probability},
    volume II, pages 267--318. University of California Press, 1972.


\bibitem{Foata}D. Foata. Etude alg\'{e}brique de certains
    probl\`emes d'analyse combinatoire et du calcul des
    probabilit\'es. \emph{Publ. Inst. Statist. Univ. Paris}, 14:81--241, 1965.

\bibitem{Goc1}V. V. Goculenko. A formula for the number of
    combinations with constrained repetitions and its
    application. \emph{Prikladnaya Diskretnaya Matematika}, 20(2):71--77,
    2013. DOI 10.17223/20710410/20/8.

\bibitem{Goc2}V. V. Goculenko. Combinatorial numbers of the
    finite multiset patitions.
    \emph{Prikladnaya Diskretnaya Matematika}, 22(4):67--72,
    2013. DOI 10.17223/20710410/22/7.

\bibitem{Graham}R. L. Graham, M. Gr\"otschel, and
    L. Lov\'asz. \emph{Handbook of Combinatorics}. Volumes 1 and 2.
    Elsevier (North-Holland), Amsterdam, and MIT Press,
    Cambridge, 1995.

\bibitem{Grin-Kleit}C. Green and D. J. Kleitman. Proof
    techniques in the theory of finite sets. In \emph{Studies in combinatorics},
    edited by G.-C. Rota, volume 17 of \emph{M. A. A. Studies in Math.}, pages 22--79. Math.
    Assoc. of Amer., Washington, DC, 1978.

\bibitem{Gr-Kn}D. H. Greene and D. E. Knuth. \emph{Mathematics
    for the analysis of algorithms}. Birkh\"{a}user, 1990.


\bibitem{Gulden}I. P. Goulden and D. M. Jackson.
    \emph{Combinatorial enumeration}. John Wiley \& Sons, 1983.


\bibitem{Hage}J. Hage. Enumerating submultisets of
    multisets. \emph{Inf. Proc. Letters}, 85(4):221--226, 2003.

\bibitem{Hall}M. Hall. \emph{Combinatorial theory}. John Wiley
    \& Sons, 1986.


\bibitem{Knuth3}D. E. Knuth. \emph{The art of computer
    pragramming}. Volume 3. Addison-Wesley, 1998.

\bibitem{MacMahon}P. A. MacMahon. \emph{Combinatory analysis}.
    Cambridge, 1915.

\bibitem{Mallows}C. L. Mallows. Some aspects of the
    random sequence. \emph{Annals of Math. Statistics}, 36:236--260, 1965.

\bibitem{Prestet}J. Prestet. \emph{\'{E}l\'{e}mens de
    Math\'{e}matiques}. Paris, 1675.

\bibitem{Riordan1}J. Riordan. \emph{An introduction to
    combinatorial analysis}. John Wiley \&
    Sons, New York, 1958.

\bibitem{Riordan2}J. Riordan. \emph{Combinatorial identities}.
    John Wiley \& Sons, 1968.


\bibitem{Ryser}H. J. Ryser. \emph{Combinatorial mathematics}.
    Mathematical Association of America, 1963.


\bibitem{Savage_Wilf}C. Savage and H. Wilf. Pattern
    avoidance in compositions and multiset permutations. \emph{Adv. in Appl.
    Math.}, 36(2):194--201, 2006.

\bibitem{Stanley1}R. P. Stanley. \emph{Enumerative
    Combinatorics}. Volume 1. Cambridge University Press, 1997.

\bibitem{Stanley2}R. P. Stanley. \emph{Enumerative
    Combinatorics}. Volume 2. Cambridge University Press, 1999.

\bibitem{Wilf H.S.}H. S. Wilf. \emph{Generatingfunctionolog}y.
    Second edition,   Academic Press, 1994.

\bibitem{MGY}R. A. Zatorskii. Counting $m$-submultisets
    through their secondary specifications. \emph{Combinatornyi
    Analiz}, 7:136--145, 1986.


\bibitem{Perest_Stud}R. A. Zatorskii. On an algorithm for
    calculation of the number of $m$-permutations on multisets.
    \emph{Matematychni Studii}, 17(2):215--219, 2002.


\end{thebibliography}
\end{document}